\documentclass{amsart}
\usepackage{amsmath,amsfonts,amssymb,amsthm,amscd}
\usepackage[small,nohug,heads=littlevee]{diagrams}

\begin{document}

\theoremstyle{definition}
\newtheorem{step}{Step}
\newtheorem{definition}{Definition}[section]
\newtheorem{ex}{Example}[section]
\theoremstyle{remark}
\newtheorem{remark}{Remark}[section]
\theoremstyle{plain}
\newtheorem{prop}{Proposition}[section]
\newtheorem{claim}{Claim}[section]
\newtheorem{lemma}{Lemma}[section]
\newtheorem{theorem}{Theorem}[section]
\newtheorem{cor}{Corollary}[section]

\title{A Universal Property of the Groups Spin$^c$ and Mp$^c$}
\author{Shay Fuchs}
\begin{abstract}
It is well known that \emph{spinors  on oriented
Riemannian manifolds cannot be defined as
sections of a vector bundle associated with the
frame bundle} (see \cite{Fr}). For this reason
spin and spin$^c$ structures are often
introduced. In this paper we prove that spin$^c$
structures have a universal property among all
other structures that enable the construction of
spinor bundles. We proceed to prove a similar
result for metaplectic$^c$ structures on
symplectic manifolds.
\end{abstract}

\maketitle

\section{Introduction}
Around 1928, while studying the motion of a free
particle in special relativity, P.A.M Dirac
raised the problem of finding a square root to
the three dimensional Laplacian, acting on smooth
functions on $\mathbb R^3$. Finding a square root
was necessary in order to study such a system in
the quantum mechanics setting. His assumptions
were that such a square root ought to be a first
order differential operator with constant
coefficients.

It is well known that in the Euclidean space
$\mathbb R^n$, the problem of finding such a
square root involves Clifford algebras and their
representations. This square root is often called
\emph{a Dirac operator}, and it acts on the
representation space for the Clifford algebra
(also called \emph{the space of spinors}).

The transition from the flat Euclidean space to a
general Riemannian manifold is not obvious. There
is no representation of the group $SO(n)$ on the
space of spinors which is compatible with the
Clifford algebra action. This means that in order
to generalize the construction of the Dirac
operator to Riemannian manifolds, we must
introduce additional structure. More precisely,
we need to lift the Riemannian structure (where
the group involved in $SO(n)$) to a `better'
group $G$.

It is known that the construction of the Dirac
operator can be carried out if our manifold has a
spin, a spin$^c$, or an almost complex structure.

In this paper we answer the question: what is the
best (or universal) structure that enables the
above process? We show that the group $Spin^c(n)$
(or a non-compact variant of it) is a universal
solution to our problem, in the sense that any
other solution will factor uniquely through the
spin$^c$ one. This suggests that spin$^c$
structures are the natural ones to consider while
quantizing the classical energy observable.

It is important to note that spin$^c$ structures
are equivalent to having an irreducible Clifford
module on the manifold. This fact appears as
Theorem 2.11 in \cite{Plymen}. This is another
hint for the universality of the group spin$^c$.

Interestingly, a similar problem can be stated in
the symplectic case, and the universal solution
involves $Mp^c$ structures, discussed in
\cite{Mp^c}. The universality statement and the
proof are almost identical to those in the
Riemannian case.

This paper is organized as follows. First we
introduce the problem of finding a square root
for the $n$-dimensional Laplacian acting on
smooth (complex-valued) functions on $\mathbb
R^n$. This will motivate the definition of
Clifford algebras and the study of their
representations. Then we generalize the problem
to an arbitrary oriented Riemannian manifold, and
explain why more structure is needed in order to
define the Dirac operator. Next, we state our
main theorem about the universality of the
(non-compact variant of the) spin$^c$ group, and
deduce a few corollaries. In the last section, we
prove a similar result in the symplectic setting.
Namely, the universality of the metaplectic$^c$
group.

This work was motivated by the introduction of
\cite{Fr}, where the problem of defining a Dirac
operator for an arbitrary oriented Riemannian
manifold is discussed, and the necessity of
additional structure is mentioned. For the study
of symplectic Dirac operators, we refer to
\cite{Sympl. Dirac}, which is the `symplectic
analogue' of \cite{Fr}. Both \cite{Fr} and
\cite{Sympl. Dirac} are excellent resources.

 \textbf{Acknowledgements.} I would like to thank Yael Karshon, for encouraging me to
 pursue
this project, guiding and supporting me through
the process of developing and writing the
material, and for having always good advice and a
lot of patience. I also would like to thank
Eckhard Meinrenken for bringing the reference
\cite{Plymen} to my attention, and for his
suggestion to check the symplectic case as well.

\section{The Euclidean case and Clifford
algebras}

Consider the negative Laplacian acting on smooth
complex valued functions on the $n$-dimensional
Euclidean space $$\Delta\colon
C^\infty(\mathbb{R}^n;\mathbb{C})\to
C^\infty(\mathbb{R}^n;\mathbb{C})\qquad,\qquad\Delta=-\sum_{j=1}^n\frac{\partial^2}{\partial
x_j^2}\quad ,$$ and suppose we are interested in
finding a square root for $\Delta$, i.e., we seek
an operator $$P\colon
C^\infty(\mathbb{R}^n;\mathbb{C})\to
C^\infty(\mathbb{R}^n;\mathbb{C})$$ with
$P^2=\Delta$. Motivated by physics, we assume
that $P$ is a first order differential operator
with constant coefficients, i.e., that
$$P=\sum_{j=1}^n\gamma_j\,\frac{\partial}{\partial
x_j}\qquad,\qquad\gamma_j\in\mathbb{C}\ .$$ A
simple computation shows that such a $P$ cannot
exists unless $n=1$, and then $P=\pm
i\frac{\partial}{\partial x}$. Indeed, the
condition $P^2=\Delta$ implies that
$$P^2=\sum_{j,l=1}^n\gamma_j\gamma_l\frac{\partial^2}{\partial
x_j\partial
x_l}=-\sum_{j=1}^n\frac{\partial^2}{\partial
x_j^2}$$ and hence
$$\gamma_j\gamma_l+\gamma_l\gamma_j=0\quad,\quad\gamma_j^2=-1\text{\quad
for all\quad} j\ne l$$ which is impossible,
unless $n=1$ and $\gamma_1=\pm i$.

One way to modify this problem is to observe that
the commutativity of complex numbers
($\gamma_j\gamma_l=\gamma_l\gamma_j$) is the
property that made this construction impossible.
Therefore, we hope to be able to find such a $P$
if the $\gamma_j$'s are taken to be matrices,
instead of complex numbers.

Thus, fix an integer $k>1$, define the
\emph{vector valued Laplacian} as
$$\underline{\Delta}\colon C^\infty(\mathbb{R}^n;\mathbb{C}^k)\to
C^\infty(\mathbb{R}^n;\mathbb{C}^k)\qquad,\qquad
\underline{\Delta}(f_1,\dots,f_k)=(\Delta
f_1,\dots,\Delta f_k) \ ,$$ and look for a square
root $$\underline P\colon
C^\infty(\mathbb{R}^n;\mathbb{C}^k)\to
C^\infty(\mathbb{R}^n;\mathbb{C}^k) \ .$$ If we
assume, as before, that $$\underline
P=\sum_{j=1}^n\gamma_j\frac{\partial}{\partial
x_j}\text{\qquad with\qquad} \gamma_j\in
M_{k\times k}(\mathbb C)\ ,$$ then $\underline
P^2=\underline\Delta$ if and only if
$$\gamma_j\gamma_l+\gamma_l\gamma_j=0\quad,\quad\gamma_j^2=-1\text{\quad
for\quad} j\ne l\ .$$

Those relations are precisely  the ones used to
define \emph{the Clifford algebra} associated to
the vector space $\mathbb{R}^n$. Here is a more
common and general definition of this concept.

\begin{definition}
For a finite dimensional vector space $V$ over
$\mathbb{K}=\mathbb{R}\text{ or }\mathbb{C}\ ,$
and a symmetric bilinear map $B\colon V\times
V\to\mathbb{K}$ , \ define \emph{the Clifford
algebra}
$$Cl(V,B)=T(V)/I(V,B)$$
where $T(V)$ is the tensor algebra of $V$, and
$I(V,B)$ is the ideal generated by $$\{v\otimes
v-B(v,v)\cdot 1\ \colon\  v\in V\}\ .$$
\end{definition}

\begin{remark}\
\begin{enumerate}
\item If $e_1,\dots,e_n$ is an orthogonal basis
for $V$, then $Cl(V,B)$ is the algebra generated
by $\{e_j\}$ with relations
$$e_je_l+e_le_j=0\quad,\quad
e_j^2=B(e_j,e_j)\text{\quad for \quad} j\ne l\
.$$
\item If $<\; ,\; >$ is the standard inner product
on $\mathbb{R}^n$, then denote
$$C_n=Cl(\mathbb{R}^n,-<\; ,\; >)\text {\quad and
\quad} C_n^c=C_n\otimes\mathbb{C}\ .$$

\end{enumerate}
\end{remark}
\begin{ex}
For $n=3$, let $$\gamma_1=\left(
                            \begin{array}{cc}
                              i & 0 \\
                              0 & -i \\
                            \end{array}
                          \right)\qquad
\gamma_2=\left(
           \begin{array}{cc}
             0 & -1 \\
             1 & 0 \\
           \end{array}
         \right)\qquad
\gamma_3=\left(
           \begin{array}{cc}
             0 & i \\
             i & 0 \\
           \end{array}
         \right)\quad.$$
Then $\{\gamma_1,\gamma_2,\gamma_3\}$ satisfy the
required relations, and $\underline
P=\gamma_1\frac{\partial}{\partial x_1}+
\gamma_2\frac{\partial}{\partial x_2}+
\gamma_3\frac{\partial}{\partial x_3}$ will be a
square root of
$\underline\Delta=-\frac{\partial^2}{\partial
x_1^2}-\frac{\partial^2}{\partial
x_2^2}-\frac{\partial^2}{\partial x_3^2}$ .\\
\end{ex}

The above discussion suggests that we look for a
representation $$\rho\colon C_n\to
End(\mathbb{C}^k)\cong M_{k\times k}(\mathbb C)$$
of the Clifford algebra $C_n$. It will be even
better if we can find an irreducible one (since
every representation of $C_n$ is a direct sum of
irreducible ones - see Proposition I.5.4 in
\cite{LM}). Once we fix such a representation, we
can set
$$\gamma_j=\rho(e_j)\qquad ,\qquad \underline
P=\sum_{j=1}^n\gamma_j\frac{\partial}{\partial
x_j}$$ where $\{e_j\}$ is the standard basis for
$\mathbb C^k$. The operator $\underline P\,$,
called \emph{a Dirac operator}, will then be a
square root of $\underline\Delta$.

Here is a known fact about representations of
complex Clifford algebras (proofs can be found in
\cite{Fr} and in \cite{LM}).

\begin{prop}
Any irreducible complex representation of $C_n$
has dimension $2^{[n/2]}$ (where $[n/2]$ is the
floor of $n/2$). Up to equivalence, there are two
irreducible representations if $n$ is odd, and
only one if $n$ is even.
\end{prop}

We conclude that finding a square root for
$\underline\Delta$ is always possible. It is
defined using an irreducible representation of
$C_n$. Note that a choice is to be made if $n$ is
odd.

\section{Manifolds}
\subsection{The problem}\label{man-intro}

We would like to generalize our previous
construction from the Euclidean space to a smooth
$n$ dimensional oriented Riemannian manifold
$(M,g)$. More generally, we look for a complex
Hermitian line bundle $S\to M$ and a smooth map
$$\rho\colon Cl(TM,-g)\to
End(S)$$ which restricts to an irreducible
representation
$$\rho_x\colon Cl(T_xM,-g_x)\to
End(S_x)$$ on the fibers of $S$. The notation
$Cl(TM,-g)$ stands for \emph{the Clifford bundle}
of $(M,g)$. That is, the vector bundle whose
fibers are the Clifford algebra of the tangent
space, with respect to the symmetric bilinear map
$-g$.

Once such a pair $(S,\rho)$ is found, we can
choose a Hermitian connection $\nabla$ on $S$,
and define a Dirac operator acting on smooth
sections of $S$, as follows. Choose a local
orthonormal frame $\{e_j\}$ and let
$$\underline P\colon
\Gamma(S)\to\Gamma(S)\qquad,\qquad \underline
P(s)=\sum_{j=1}^n\rho(e_j)\left[\nabla_{e_j}s\right]\
.$$ It turns out that $\underline P$ is
independent of the local frame, and thus gives
rise to a globally defined operator on sections
of $S$.

\subsection{The search for the vector bundle $S$}
If no additional structure is introduced on our
manifold $M$, then we may try to construct the
vector bundle $S$ as an associated bundle to the
bundle $SOF(M)$ of oriented orthonormal frames on
$M$. This means that we try to take
$$S=SOF(M)\times_{SO(n)}\mathbb C^k$$
where $k=2^{[n/2]}$ and $SO(n)$ acts on $\mathbb
C^k$ via a representation $$\epsilon\colon
SO(n)\to End(\mathbb C^k)\ .$$ We can use an
irreducible representation $$\rho\colon C_n\to
End(\mathbb C^k)$$ in order to define an
 action of the Clifford bundle
$Cl(TM)$ on $S$ as follows.\\
For any $x\in M$, $\alpha\in T_xM\subset
Cl(T_xM,-g_x)$, $v\in\mathbb C^k$ and a frame
$f\colon\mathbb R^n\xrightarrow{\simeq} T_xM$ in
$SOF_x(M)$, let $\alpha$ act on $[f,v]\in S_x$ by
\begin{equation*}(\alpha,[f,v]) \mapsto
[f\,,\,\rho(f^{-1}(\alpha))v] \ . \end{equation*}\\
This will be a well defined action on $S$ if and
only if
$$[f\circ A\,,\,\rho((f\circ
A)^{-1}\alpha)v]=[f,\rho(f^{-1}(\alpha))(\epsilon(A)v)]$$
for any $A\in SO(n)$. This is equivalent to
$$\epsilon(A)\circ
\rho(A^{-1}y)=\rho(y)\circ\epsilon(A)$$ where
$y=f^{-1}(\alpha)$, and is an equality between
linear endomorphisms of $\mathbb C^k$. To
summarize, we look for a representation
$\epsilon\colon SO(n)\to End(\mathbb C^k)$ with
the property that
\begin{equation}\label{eqn-so(n)}
\rho(Ay)=\epsilon(A)\circ\rho(y)\circ\epsilon(A)^{-1}
\end{equation}
for all $A\in SO(n)\text{ and }y\in\mathbb R^n$.

Unfortunately:
\begin{claim}
For $n\ge 3$, there is no representation
$\epsilon\colon SO(n)\to End(\mathbb C^k)$
satisfying Equation~(\ref{eqn-so(n)}) for all $A$
and $y$.
\end{claim}

The proof  will follow from a more general
statement later (see Claim \ref{claim-exmps}).

\subsection{Introducing additional structure}
It seems that in order to construct a vector
bundle $S$ over an $n$ dimensional oriented
Riemannian manifold $(M,g)$, on which the
Clifford bundle $ Cl(TM)$ acts irreducibly, we
will have to introduce new structure on our
manifold: we will need to lift the structure
group from $SO(n)$ to a `better' group $G$. Here
is what we mean by \emph{lifting the structure
group}.

\begin{definition}
For an $n$ dimensional oriented Riemannian
manifold $(M,g)$, a \emph{lifting of the
structure group to a Lie group $G$} is a
principal $G$-bundle $\pi\colon P\to M$, together
with a group homomorphism $p\colon G\to SO(n)$
and a smooth map $\pi_1\colon P\to SOF(M)$ such
that
$$\pi_1(x\cdot g)=\pi_1(x)\cdot
p(g)\qquad\text{for}\qquad x\in P\ ,\ g\in G\ ,$$
and such that $\pi=\pi_2\circ\pi_1$ (where
$\pi_2\colon SOF(M)\to M$ is the projection).
\end{definition}

In other words, we require that the following
diagram will commute, and~$\pi=\pi_2\circ\pi_1$.
$$\begin{CD}
P@<<<P\times G\\
@V\pi_1 VV @VV \pi_1\times p V\\
SOF(M)@<<<SOF(M)\times SO(n)\\
@V \pi_2 VV\\
M\\[15pt]
\end{CD} $$

Once we have such a lift, we can try to construct
our vector bundle of rank $k=2^{[n/2]}$ as
$$S=P\times_G\mathbb C^k\ ,$$
where $G$ acts on $\mathbb C^k$ via a
representation $\epsilon\colon G\to End(\mathbb
C^k)$.\\
This will work if the action of $Cl(TM)$ on $S$,
given by
$$(\alpha,[\tilde f,v])\mapsto[\tilde f,\rho(f^{-1}(\alpha))v]\ ,$$
is well defined. Here $\alpha\in Cl(T_xM)$,
$\tilde f\in P_x$, $v\in\mathbb C^k$, and
$f=p(\tilde f)\colon \mathbb R^n\to T_xM$ is a
frame in $SOF_x(M)$. As before, $\rho$ is a fixed
irreducible complex representation of $C_n$.

Equation (\ref{eqn-so(n)}), which state the
condition $\epsilon$ has to satisfy, becomes
\begin{equation}\label{eqn-G}
\rho(p(A)y)=\epsilon(A)\circ\rho(y)\circ\epsilon(A)^{-1}
\end{equation}
for all $y\in\mathbb R^n$ and $A\in G$.

To summarize, we look for a Lie group $G$, and a
representation $\epsilon\colon G\to End(\mathbb
C^k)$ for which Equation (\ref{eqn-G}) is
satisfied for all $A$ and $y$.

\section{The main theorem}
Given  an irreducible representation $\rho\colon
C_n\to End(\mathbb C^k)$ ($k=2^{[n/2]}$), we look
for a Lie group $G$, a representation
$\epsilon\colon G\to End(\mathbb C^k)$, and a
homomorphism $p\colon G\to SO(n)$ for which
Equation (\ref{eqn-G}) is satisfied. Note that
this problem is of algebraic flavour and does not
involve the manifold, the tangent bundle or the
Clifford bundle. Thus, our problem is reduced to
one where the unknowns are a Lie group, a
representation and a group homomorphism.

As we will see, there is more than one solution
to this problem, but only one (up to a certain
equivalence) which is universal in the sense that
every other solution will factor through the
universal one. In this universal solution, the
group is $$G=\left(Spin(n)\times\mathbb
C^\times\right)/K$$ where $Spin(n)$ is the double
cover of $SO(n)$ and $K$ is the two element
subgroup generated by $(-1,-1)$. This is a
noncompact group.

Another way to define this group is as the set of
all elements of the form
$$c\cdot v_1v_2\cdots v_l\in C_n^c=C_n\otimes\mathbb C$$
where $c\in\mathbb C^\times$, $l\ge 0$ is even,
and each $v_j\in\mathbb R^n$ is of (Euclidean)
norm 1.

For each element $x\in G$ and $y\in\mathbb
R^n\subset\mathbb C_n^c$, we have $Ad_x(y)=x\cdot
y\cdot x^{-1}\in\mathbb R^n$, and the map
$$y\in\mathbb R^n\mapsto Ad_x(y)\in\mathbb R^n$$
is in $SO(n)$. This defines a group homomorphism
$$\lambda^c\colon G\to SO(n)\quad,\quad x\mapsto Ad_x$$
(see \cite{Fr} for details).

Finally, note that any $B\in SO(n)$ acts on the
Clifford algebra $C_n^c$ in a natural way. This
action is induced from the standard action of
$SO(n)$ on $\mathbb R^n$. Furthermore, Equation
(\ref{eqn-G}) is satisfied for all $y\in\mathbb
R^n$ and $A\in G$ if and only if it is satisfied
for all $y\in C_n^c$ and $A\in G$.

Now we can state the universality property of the
group $G$.
\begin{theorem}\label{main-thm}
Fix an irreducible complex representation $\rho$
of $C_n^c$, and let $k=2^{[n/2]}$. Then:
\begin{enumerate}
\item For  $G=\left(Spin(n)\times\mathbb C^\times\right)/K$, $p=Ad\colon G\to SO(n)$
and $\epsilon=\rho|_G\colon G\to~ End(\mathbb
C^k)$, we have
$$\rho(p(A)y)=\epsilon(A)\circ\rho(y)\circ\epsilon(A)^{-1}$$
for all $y\in C_n^c$ and $A\in G$.
\item If $G'$ is a Lie group, $p'\colon G'\to
SO(n)$ a group homomorphism, and $\epsilon'\colon
G'\to End(\mathbb C^k)$  a representation, such
that
$$\rho(p'(A)y)=\epsilon'(A)\circ\rho(y)\circ\epsilon'(A)^{-1}$$
for all $y\in C_n^c$ and $A\in G'$, then there is
a unique homomorphism $f\colon G'\to G$ such that
$$ p'=p\circ f\qquad\text{and}\qquad
\epsilon'=\epsilon\circ f\ .$$
\end{enumerate}
\end{theorem}

\begin{remark}\
\begin{enumerate}
\item The group $G$ acts on $C_n^c$ via
$$(A,y)\mapsto p(A)y$$ and on $End(\mathbb C^k)$
via $$(A,\varphi)\mapsto \epsilon
(A)\circ\varphi\circ\epsilon(A)^{-1}\ .$$
Therefore, in part (1) of the theorem we claim
that $\rho$ is $G$-equivariant.
\item Part (2) of the theorem implies that the
the following two diagrams are commutative.

{
\begin{diagram}
        &       &G      &&   &       &       &G\\
        &\ruDashto^{f}  &\dTo_{p}           &&   &       &\ruDashto^{f}  &\dTo_{\varepsilon}\\
G'       &\rTo_{p'}   &SO(n)          &&   &G'      &\rTo_{\varepsilon' \ \ }   &End(\mathbb{C}^k)\\
\end{diagram}
}
\end{enumerate}

\end{remark}

\begin{proof} \
\begin{enumerate}
\item For any $A\in G$ and $y\in C_n^c$
we have
\begin{align*}
\rho(p(A)y)&=\rho(Ad_A(y))=\rho(A\cdot y\cdot
A^{-1})=\rho(A)\cdot\rho(y)\cdot\rho(A^{-1})=\\
&=\epsilon(A)\circ\rho(y)\circ\epsilon(A)^{-1}\hfill
\end{align*}
\item Fix an element $g\in G'$, and choose an
element $A\in Spin(n)$ for which
$p(A)=Ad_A=p'(g)$. We claim that the endomorphism
$$\rho(A^{-1})\circ\epsilon'(g)\colon\mathbb
C^k\to\mathbb C^k$$ is a nonzero (complex)
multiple of the identity. To see this, start from
the given equality
$$\rho(p'(g)y)=\epsilon'(g)\circ\rho(y)\circ\epsilon'(g)^{-1}$$
which is equivalent to
$$\epsilon'(g)\circ\rho(y)=
\rho(Ad_A(y))\circ\epsilon'(g)$$ and to
$$\left[\rho(A^{-1})\circ\epsilon'(g)\right]\circ\rho(y)=\rho(y)
\circ\left[\rho(A^{-1})\circ\epsilon'(g)\right]
$$ for all  $y\in C_n^c$.

It is known that any irreducible complex
representation of $C_n^c$ must be onto, and hence
the last equality means that
$\rho(A^{-1})\circ\epsilon'(g)$ commutes with all
endomorphisms of $\mathbb C^k$, and thus must be
a multiple of the identity (it is nonzero since
it is invertible). Write
$\rho(A^{-1})\circ\epsilon'(g)=c\cdot I$ for
$c\in \mathbb C^\times$ and define
$$f\colon G'\to G\qquad,\qquad g\mapsto [A,c]\in
G\ .$$ This map is a well defined. It is a group
homomorphism since if $g_j\in G'$ (for $j=1,2$),
$A_j\in Spin(n)$ with $p'(g_j)=Ad_{A_j}$ and
$c_j\in\mathbb C^\times$ satisfy
$$c_j\cdot
I=\rho(A_j^{-1})\circ\epsilon'(g_j)$$ then we
have $$p(g_1g_2)=Ad_{A_1A_2}$$ and
$$ c_1c_2\cdot
I=\rho((A_1A_2)^{-1})\circ\epsilon'(g_1g_2)\ .$$
This implies that $f(g_1g_2)=f(g_1)f(g_2)$.

Also we have
$$p'(g)=Ad_A=Ad_{c\cdot A}=p(f(g))$$
and
$$\epsilon'(g)=c\cdot\rho(A)=\epsilon(c\cdot
A)=\epsilon(f(g))$$ for all $g\in G'$ as needed.

It is not hard to see that our construction
implies also the uniqueness of the map $f$. After
all, if such an $f$ exists, and for $g\in G'$ we
write $f(g)=[A,c]\in G$ , then $p(A)=Ad_A=p'(g)$,
which means than $A$ is determined up to sign.
Furthermore, the relation
$\epsilon'(g)=\epsilon(f(g))$ implies that
$$\rho(A^{-1})\circ\epsilon'(g)=\epsilon([1,c])=c\cdot
I\ ,$$ which determines the value of $c$.
Therefore $f(g)=[A,c]$ is uniquely determined by
our conditions.

\end{enumerate}
\end{proof}

\begin{remark} \
\begin{enumerate}
\item The triple $(G=\left(Spin(n)\times\mathbb
C^\times\right)/K,p,\epsilon)$ is the only
universal solution up to equivalence. More
precisely, if $(G',p',\epsilon')$ is another
universal solution, then there is a unique
isomorphism $\varphi\colon G'\to G$ satisfying
$\epsilon'=\epsilon\circ\varphi$ and
$p'=p\circ\varphi$.
\item There is a natural Hermitian product on the
representation space $\mathbb C^k$ with respect
to which $\rho$ is unitary. If we require that
$\epsilon'$ will be unitary, then universal
solution will involve the (compact) group
$$Spin^c(n)=\left(Spin(n)\times U(1)\right)/K \ .$$
The universality statement and the proof are
almost identical to the noncompact case.
\item It is important to note that although the
Dirac operator is a square root of the Laplacian
in the case of $\mathbb R^n$, this is no longer
true in the manifold case. The Dirac operator,
whose definition was outlined in
\S\ref{man-intro}, will be related to the
Laplacian via the Schr\"{o}dinger-Lichnerowicz
formula, which involves the scalar curvature of
the manifold, and the curvature of the Hermitian
connection on the vector bundle $S$ (see \S3.3 in
\cite{Fr}).
\end{enumerate}
\end{remark}

\section{Some corollaries}
Denote again by $\rho\colon C_n\to End(\mathbb
C^k)$ ($k=2^{[n/2]}$) an irreducible
representation, $G=\left(Spin(n)\times\mathbb
C^\times\right)/K$, and by $p\colon G\to SO(n)$
the natural homomorphism. Then Theorem
\ref{main-thm} implies the following.

\begin{cor}\label{cor1}
A Lie group $G'$ and a homomorphism $p'\colon
G'\to SO(n)$ give rise to a bijection between
\begin{multline*}
\mathcal{A}=\left\{
\begin{array}{c}
  \text{Representations }\epsilon'\colon G'\to
End(\mathbb C^k)\text{ satisfying }\\
\epsilon'(g)\circ\rho(y)=\rho(p'(g)y)\circ\epsilon'(g)
\text{ for all } g\in G',\ y\in\mathbb R^n
\end{array}
\right\}\hfill\\
\end{multline*}

and

\begin{multline*}
\mathcal B=\left\{ \text{ Homomorphisms }f\colon
G'\to G\text { such that } p'=p\circ f \
\right\}\hfill
\end{multline*}
\end{cor}

\begin{proof}
Part (2) in Theorem \ref{main-thm} provides a
function $f\in\mathcal B$ for every
$\epsilon'\in\mathcal A$. Conversely, if
$f\in\mathcal B$, then $\epsilon'=\epsilon\circ
f$ is in $\mathcal A$, by part (1) of Theorem
\ref{main-thm}.
\end{proof}

The above corollary provides an easy criterion
for checking whether a lifting of the structure
to a group $G'$ will enable us to construct an
irreducible Clifford bundle action on a vector
bundle associated with this lifting. We give a
few examples in the following claim.

\begin{claim}\ \label{claim-exmps}
 If $G'=U(n/2)$ (for an even $n$) or
$G'=Spin(n)$,  then there exist a homomorphism
$p'\colon G'\to SO(n)$ and a representation
$\epsilon'\colon G'\to End(\mathbb C^k)$ for
which
$\epsilon'(g)\circ\rho(y)=\rho(p'(g)y)\circ\epsilon'(g)$
 for all $g\in G',\ y\in\mathbb R^n$.\\
 If $G'=SO(n)$ ($n\ge 3$) and $p'\colon G'\to SO(n)$ is the
  identity, then there is no $\epsilon'$
  satisfying the latter equality.
\end{claim}

\begin{proof}
For $G'=Spin(n)$ take $p'$ to be the double cover
of $SO(n)$, $f\colon Spin(n)\to G$ the inclusion,
and use Corollary \ref{cor1}.\\
 For $G'=U(n/2)$,
take $p'$ to be the standard inclusion
$U(n/2)\subset SO(n)$. It is possible to define
$f\colon U(n/2)\to G$ such that $p'=p\circ f$
(see page 27 in \cite{Fr}). By the above
corollary, the conclusion follows.

Finally, for $G'=SO(n)$ and $p'=Id$, if such an
$\epsilon'$ would exist, the corollary implies
that there is an $f\colon SO(n)\to G$ for which
$p\circ f=Id$. This is impossible since the
fundamental group of $SO(n)$ is $\mathbb Z_2$ and
of $G$ is $\mathbb Z$.
\end{proof}

The above claim implies some well known facts:
Every spin and every almost complex manifold is
also a spin$^c$ manifold in a natural way. Also,
an irreducible Clifford module cannot be defined
as a tensor bundle (i.e., as a vector bundle
associated with the frame bundle of the
manifold).

\section{The Symplectic case}
For the symplectic group, a similar problem can
be stated. The universal group in this case will
be the complexified metaplectic group
$Mp^c(n)=Mp(n)\times_{\mathbb Z_2}U(1)$, if we
demand unitary representations, or
$Mp(n)\times_{\mathbb Z_2}\mathbb C^\times$
otherwise. In this section we outline the setting
in this case, and prove a similar universality
statement.

\subsection{Symplectic Clifford algebras}
Let $V$ be a real vector space of dimension $2n$.
If $B\colon V\times V\to\mathbb R$ is a symmetric
bilinear form on V, then the ideal (in the tensor
algebra $T(V)$) generated by expressions of the
form
\begin{equation} v\cdot v-B(v,v)\cdot
1\qquad,\qquad v\in V \end{equation} is the same
one generated by
\begin{equation}\label{rel-omega} v\cdot u+u\cdot
v-2\cdot B(u,v)\qquad,\qquad u,v\in V\
.\end{equation} Suppose now that $\omega\colon
V\times V\to\mathbb R$ is a symplectic (i.e., an
antisymmetric and non-degenerate bilinear) form
on $V$. Since $\omega(v,v)=0$ for all $v\in V$,
we would better modify (\ref{rel-omega}) and
define the symplectic Clifford algebra as
follows. We follow \cite{Sympl. Dirac} and omit
the coefficient `2' in our definition.

\begin{definition}
\emph{The symplectic Clifford algebra} associated
with the symplectic vector space $(V,\omega)$ is
defined as
$$Cl^s(V,\omega)=T(V)/I(V,\omega)$$
where $I(V,B)$ is the ideal generated by
$$\{v\cdot w-w\cdot v+\omega(v,w)\cdot 1:v,w\in V\}\ .$$
\end{definition}

\begin{remark}
If $V=\mathbb R^{2n}$ and $\omega$ is the
standard symplectic form, given by
$$\omega(x,y)=\sum_{j=1}^n x_jy_{n+j}-x_{n+j}y_j\quad,\quad x,y\in\mathbb R^{2n}\ ,$$ then we
denote $$Cl^s_n=Cl^s(\mathbb R^{2n},\omega)
\quad\mbox{and}\quad \mathbb
Cl^s_n=Cl^s_n\otimes\mathbb C\ .$$ The symplectic
Clifford algebra in this case is also called the
\emph{Weyl algebra}, and is useful since its
generators satisfy relations which are similar to
the relations satisfied by the position and
momentum operators in quantum mechanics (see
\S1.4 is \cite{Sympl. Dirac}).
\end{remark}

Denote by $\mathcal{S}(\mathbb R^n)$ the Schwartz
space of rapidly decreasing complex-valued smooth
functions on $\mathbb R^n$. If $e_1,\dots,e_{2n}$
is the standard basis for $\mathbb R^{2n}$, then
define a linear action $$\rho\colon\mathbb
R^{2n}\to End(\mathcal{S}(\mathbb R^n))$$ by
assigning
$$ \rho(e_j)f=i\cdot x_jf\qquad\mbox{for}\qquad
j= 1,\dots,n$$
$$\rho(e_j)f=\frac{\partial f}{\partial
x_j}\qquad\mbox{for}\qquad j=n+1,\dots,2n$$ and
extend by linearity. This action extends (see
\S1.4 in \cite{Sympl. Dirac}) to a linear map
$$Cl^s_n\to End(\mathcal S(\mathbb R^n))$$ which
is \emph{not} an algebra homomorphism.

For each $v\in\mathbb R^{2n}$, $\rho(v)$ can be
regarded as a continuous operator on the Schwartz
space. We call the map $\rho$ \emph{Clifford
multiplication}.

\subsection{The metaplectic representation} The
metaplectic group $Mp(n)$ will play the role that
the spin group $Spin(n)$ played in the Riemannian
case. The symplectic group is
$$Sp(n)=\{A\in
GL(2n,\mathbb R):\omega(Av,Aw)=\omega(v,w)\}$$
where $\omega$ is the standard symplectic form on
$\mathbb R^{2n}$. This is a connected and
non-compact Lie group.

The fundamental group of $Sp(n)$ is isomorphic to
$\mathbb Z$, and thus $Sp(n)$ has a unique
connected double cover, which is denoted by
$Mp(n)$. Denote by $$p\colon Mp(n)\to Sp(n)$$ the
covering map, and by $-1\in Mp(n)$ the nontrivial
element in $Ker(p)$.

Define $$G=\left(Mp(n)\times\mathbb
C^\times\right)/K$$ where $K=\{(1,1),(-1,-1)\}$.
The covering map extends to a map (also denoted
by $p$) $$G\to Sp(n)\qquad,\qquad [A,z]\mapsto
p(A) \ .$$

There is an important infinite dimensional
unitary representation of the metaplectic group
on the Hilbert space $L^2(\mathbb R^n)$, which is
called \emph{the metaplectic representation}. We
denote it by $$\mathfrak{m}\colon Mp(n)\to
U(L^2(\mathbb R^n))$$ where $U(L^2(\mathbb R^n))$
is the group of unitary operators on $L^2(\mathbb
R^n)$. For the construction of $\mathfrak m$, see
\cite{Sympl. Dirac} and references therein. This
representation has many interesting properties,
but all we need here is the facts that the
Schwartz space $\mathcal S(\mathbb R^n)\subset
L^2(\mathbb R^n)$ is an invariant subspace for
$\mathfrak m$, and is dense in $L^2(\mathbb
R^n)$.

We extend $\mathfrak m$ to a representation of
the group $G=(Mp(n)\times\mathbb C^\times)/K$, by
$$\epsilon\colon G\to End(L^2(\mathbb
R^n))\qquad,\qquad
\epsilon([A,z])=z\cdot\mathfrak m(A)\ .$$

\subsection{The universality of the metaplectic
group} Now we can state the universality theorem
(for the group $G$), which turns out to be almost
identical to the corresponding theorem in the
Riemannian case.

\begin{theorem}\label{main-thm-sym}
Let $\rho\colon \mathbb R^{2n}\to End(\mathcal
S(\mathbb R^n))$ be the Clifford multiplication
map. Then:
\begin{enumerate}
\item For  $G=\left(Mp(n)\times\mathbb C^\times\right)/K$, $p\colon G\to
Sp(n)$ and $\epsilon\colon G\to~ End(L^2(\mathbb
R^n))$ defined above, we have
$$\rho(p(A)y)=\epsilon(A)\circ\rho(y)\circ\epsilon(A)^{-1}$$
for all $y\in \mathbb R^{2n}$ and $A\in G$ (i.e.,
$\rho$ is $G$-equivariant).\\This is an equality
of operators on the Schwartz space $\mathcal
S(\mathbb R^n)$.\\
\item If $G'$ is a Lie group, $p'\colon G'\to
Sp(n)$ a group homomorphism, and $\epsilon'\colon
G'\to End(\mathcal S(\mathbb R^n))$  a continuous
representation, such that
$$\rho(p'(A)y)=\epsilon'(A)\circ\rho(y)\circ\epsilon'(A)^{-1}$$
for all $y\in \mathbb R^{2n}$ and $A\in G'$, then
there is a unique homomorphism $f\colon G'\to G$
such that
$$ p'=p\circ f\qquad\text{and}\qquad
\epsilon'=\epsilon\circ f\ .$$
\end{enumerate}
\end{theorem}

\begin{proof} \
\begin{enumerate}
\item For $Mp(n)$, this is proved in Lemma
1.4.4 in \cite{Sympl. Dirac}. The proof for $G$
follows.

\item To prove the second part, we follow the
same idea as in the Riemannian case. Fix an
element $g\in G'$, and choose an element $A\in
Mp(n)$ for which $p(A)=p'(g)$. We show that the
endomorphism
$$ D=\epsilon
(A^{-1})\circ\epsilon'(g)\colon\mathcal S(\mathbb
R^n)\to\mathcal S(\mathbb R^n)$$ is a nonzero
complex multiple of the identity. Once this is
done, the rest of the proof will be identical to
the proof of Theorem \ref{main-thm}, part (2).

By assumption, we have

$$\epsilon(A)\circ\rho(y)\circ\epsilon(A)^{-1}=
\epsilon'(g)\circ\rho(y)\circ\epsilon'(g)^{-1}$$
which is equivalent to
$$\rho(y)\circ D=D\circ\rho(y) $$
for all $y\in\mathbb R^{2n}$. From the definition
of $\rho$ we conclude that $D$ is a continuous
operator on the Schwartz space $\mathcal
S(\mathbb R^n)$ which commutes with all
multiplication and derivative operators:
$$f\mapsto x_j\cdot f\qquad\mbox{and}\qquad
f\mapsto\frac{\partial f}{\partial x_j}\ .$$ Such
an operator must be a complex multiple of the
identity. This follows from the fact that the map
$\rho$ gives rise to an irreducible
representation of the symplectic Clifford algebra
$Cl_n^s$ on the space $L^2(\mathbb R^n;\mathbb
C)$. For a proof of this fact for $n=1$ see
Theorem 3 (page 44) in \cite{Kirillov}. The
$n$-dimesional case follows.

\end{enumerate}
\end{proof}

\begin{remark}
As in the Riemannian case, if we require that
$\epsilon'$ will be a unitary representation,
then the group $G$ in Theorem \ref{main-thm-sym}
will be replaced with $\left(Mp(n)\times\mathbb
U(1)\right)/K$.
\end{remark}

\begin{remark}
The construction of a Dirac operator in the
Riemannian case was motivated by the search for a
square root for the (negative) Laplacian. One may
wonder what is the symplectic analog of the Dirac
and the Laplacian operators. In Chapter 5 of
\cite{Sympl. Dirac} symplectic Dirac and
associated second order operators are discussed.
However, it is not clear to me if the
\emph{search for a square root} in the Riemannian
case has a (satisfactory) symplectic analogue.
\end{remark}

\end{document}